\documentclass[reqno,12pt]{amsart}
\usepackage[body={17cm,21cm}]{geometry}
\usepackage{amsmath}
\usepackage{amsfonts}
\usepackage{amssymb}
\usepackage{graphicx}
\usepackage{amsfonts}
\usepackage{amssymb}
\usepackage{mathrsfs}
\usepackage[french]{babel}
\usepackage{amscd,latexsym,amsthm,amsfonts,amssymb,amsmath,amsxtra}
\usepackage[colorlinks, urlcolor=blue, bookmarks=true,bookmarksopen=true,
citecolor=blue]{hyperref}
\usepackage[all]{xy}%
\providecommand{\U}[1]{\protect\rule{.1in}{.1in}}
\RequirePackage{amsmath}
\RequirePackage{amssymb}

\newcommand{\et}{{\rm{\acute et}}}
\newcommand{\BA}{{\mathbb {A}}}

\newcommand{\BC}{{\mathbb {C}}}

\newcommand{\BG}{{\mathbb {G}}}

\newcommand{\BP}{{\mathbb {P}}}
\newcommand{\BQ}{{\mathbb {Q}}}

\newcommand{\BZ}{{\mathbb {Z}}}

\newcommand{\CC}{{\mathcal {C}}}

\newcommand{\CH}{{\mathcal {H}}}

\newcommand{\CM}{{\mathcal {M}}}

\newcommand{\Br}{{\mathrm{Br}}}

\newcommand{\Pic}{{\mathrm{Pic}}}

\newcommand{\Ker}{{\mathrm{Ker}}}

\renewcommand{\Im}{{\mathrm{Im}}}

\newcommand{\ra}{\rightarrow}
\newcommand{\iso}{\stackrel{\sim}{\rightarrow} }

\newcommand{\sbt}{\subset}
\newcommand{\bk}{\bar{k}}

\numberwithin{equation}{section}

\theoremstyle{remark}

\newtheorem{defi}{\rm{\textbf{D\'efinition}}}[section]

\newtheorem{exam}[defi]{\rm{\textbf{Example}}}

\newtheorem{rem}[defi]{\rm{\textbf{Remarque}}}

\theoremstyle{plain}

\newtheorem{thm}[defi]{\rm{\textbf{Th\'eor\`eme}}}

\newtheorem{cor}[defi]{\rm{\textbf{\textbf{Corollaire}}}}

\newtheorem{lem}[defi]{\rm{\textbf{Lemme}}}

\newtheorem{prop}[defi]{\rm{\textbf{\textbf{Proposition}}}}

\begin{document}

\title[]
{Cohomologie non ramifi\'ee de degr\'e 3 : vari\'et\'es cellulaires et   surfaces de del Pezzo de  degr\'e au moins 5}

\author{Yang CAO}

\address{Yang CAO \newline Laboratoire de Math\'ematiques d'Orsay
\newline Univ. Paris-Sud, CNRS, Univ. Paris-Saclay \newline 91405 Orsay, France}

\email{yang.cao@math.u-psud.fr}

\date{\today.}

\maketitle

\begin{abstract}
 Dans cet article,  o\`u le corps de base est un corps  
 de caract\'eristique z\'ero quelconque, pour $X$
 une vari\'et\'e g\'eom\'etriquement cellulaire, on \'etudie 
  le quotient du  troisi\`eme groupe de cohomologie non ramifi\'ee
 $H^3_{nr}(X,\BQ/\BZ(2))$ par sa partie constante.
 Pour $X$ une compactification lisse d'un torseur universel
 sur une surface g\'eom\'etriquement rationnelle, on montre que ce quotient est fini.
 Pour $X$ une surface de  del Pezzo de degr\'e $\geq 5$, on montre que ce
  quotient est trivial, sauf si $X$ est
une surface de del Pezzo de degr\'e 8 d'un type particulier.
  
 We consider geometrically cellular varieties  $X$ over an arbitrary field
 of characteristic zero. We study the quotient of the  third unramified cohomology group
  $H^3_{nr}(X,\BQ/\BZ(2))$ by its constant part.
  For $X$ a smooth compactification of a universal torsor over a geometrically
  rational surface, we show that this quotient if finite.
  For $X$ a del Pezzo  surface  of degree  $\geq 5$,
  we show that this quotient is zero, unless $X$
  is a  del Pezzo surface of degree 8 of a special type.
  \end{abstract}

\tableofcontents

\section{Introduction}

Soient $k$ un corps de caract\'eristique 0,   ${\bk}$ une cl\^oture alg\'ebrique et ${\Gamma_k}$ le groupe de Galois de  ${\bk}$ sur $k$.
 Pour une vari\'et\'e lisse $X$ sur $k$ et un faisceau \'etale $F$ sur $X$, on rappelle que \emph{la cohomologie non ramifi\'ee} de $X$ de degr\'e $n$ est le groupe 
$$H_{nr}^n(X,F):=H^0_{Zariski}(X,\CH^n(X,F)),$$
 o\`u $\CH^n(X,F)$ est le faisceau Zariski associ\'e au pr\'efaisceau $\{ U\sbt X\}\mapsto H^n_{\et}(U,F)$.
 Soit $F=\BQ/\BZ(j)$ le faisceau des racines de l'unit\'e tordu $j$ fois.
 Les  groupes $H_{nr}^n(X,\BQ/\BZ(j))$ sont des invariants $k$-birationnels des $k$-vari\'et\'es projectives lisses
 g\'eom\'etriquement connexes, r\'eduits \`a $H^n(k,  \BQ/\BZ(j))$ pour
 $X$ 
$k$-rationnelle,   c'est-\`a-dire $k$-birationnelle \`a un espace projectif.
(cf. \cite[Th\'eor\`eme 4.1.1 et Proposition 4.1.4]{CT95}).
Le groupe $H_{nr}^2(X,\BQ/\BZ(1))$ n'est autre que le groupe de Brauer de $X$, il a \'et\'e fort \'etudi\'e.
On s'est int\'eress\'e plus r\'ecemment au   groupe $H_{nr}^3(X,\BQ/\BZ(2))$. Le cas des coniques fut trait\'e par Suslin.
En dimension quelconque,
le quotient
$H_{nr}^3(X,\BQ/\BZ(2))/H^3(k,  \BQ/\BZ(2))$ est trivial pour  toute  quadrique lisse qui n'est pas une quadrique d'Albert
(Kahn, Rost, Sujatha, voir \cite[Thm 10.2.4 (b)]{Ka00}).

Dans cet article, nous nous int\'eressons aux surfaces g\'eom\'etriquement rationnelles les plus simples,
les  surfaces de del Pezzo de degr\'e au moins 5. Rappelons que l'indice $I(X)$ d'une $k$-vari\'et\'e $X$  est le pgcd des degr\'es sur $k$ des points ferm\'es. Si une surface de del Pezzo $X$
de degr\'e au moins 5  a un indice $I(X)=1$, alors elle a un $k$-point et elle est $k$-rationnelle (cf. Thm. \ref{t2}). On a donc alors
  $H^3(k,\BQ/\BZ(2))= H_{nr}^3(X,\BQ/\BZ(2))$.
  
  Nous nous int\'eressons ici au cas o\`u $X(k)$ est \'eventuellement vide.
Nous montrons (Th\'eor\`eme \ref{t1}) :

\begin{thm} Soit $X$ une $k$-surface de del Pezzo de degr\'e $\geq 5$.
Alors $\frac{H^3_{nr}(X,\BQ/\BZ(2))}{H^3(k,\BQ/\BZ(2))}=0$,
 sauf peut-\^etre si  $deg(X)=8$, $I(X)=4$ et  il existe des  coniques lisses $C_1$, $C_2$ sur $k$ telles que $X\iso C_1\times C_2$.
\end{thm}

On construit une surface de del Pezzo  $X$ de degr\'e 8 sur
le corps $k=\BC(x,y,z)$ pour laquelle $\frac{H^3_{nr}(X,\BQ/\BZ(2))}{H^3(k,\BQ/\BZ(2))} \neq 0$
(Exemple   \ref{exampleH3nontrivial}).

\medskip

Pour les surfaces g\'eom\'etriquement rationnelles g\'en\'erales, nous montrons (Th\'eor\`eme \ref{torseuruniv1}):

\begin{thm} \label{torseuruniv}
Soit $X$ une $k$-surface projective, lisse, g\'eom\'etriquement rationnelle.
Soit $\mathcal{T} \to X$ un torseur universel sur $X$ et soit $\mathcal{T}^c$
une $k$-compactification lisse de  $\mathcal{T}$. Alors  le groupe
$H^3_{nr}(\mathcal{T}^c,\BQ/\BZ(2))/H^3(k, \BQ/\BZ(2))$ est fini.
\end{thm}

\medskip

Pour le faisceau $\BZ/n(i)= \mu_{n}^{\otimes i}$ ou pour le complexe de faisceau $\BZ(i)$ dont la d\'efinition est rappel\'ee plus bas,
on note  $H^j(-,-)$ la cohomologie   \'etale. 
Pour une courbe conique  lisse $C$ sur $k$, on note $[C]\in \Br(k)$ sa classe dans le groupe de Brauer de $k$.

\section{Sur les vari\'et\'es cellulaires et leur cohomologie non ramifi\'ee}

On rappelle la d\'efinition d'une  vari\'et\'e cellulaire  \cite[D\'efinition 3.2]{K97}.

\begin{defi}
Un $k$-sch\'ema de type fini $X$ \emph{a une d\'ecomposition cellulaire} (bri\`evement: est \emph{cellulaire}) s'il existe un sous-ensemble ferm\'e propre $Z\sbt X$ tel que $X\setminus Z$ est isomorphe \`a un espace affine et $Z$ a une d\'ecomposition cellulaire.

Un $k$-sch\'ema de type fini $X$ est dit \emph{g\'eom\'e\-triquement cellulaire} si $X_{\bk}$ a une d\'e\-composition cellulaire.
\end{defi} 

\begin{prop}\label{universelcellulaire}
Soit $k$  un corps alg\'ebriquement clos.

 (1) Une surface projective, lisse, $k$-rationnelle est cellulaire.

(2) Une vari\'et\'e torique, lisse, projective sur $k$ est cellulaire.

(3) Soient $T$ un tore sur $k$ et $T^c$ une $T$-vari\'et\'e torique, lisse, projective.
Soient $X$ une vari\'et\'e cellulaire sur $k$ et  $Y\ra X$ un $T$-torseur.
 Alors $Y^c:=Y\times^TT^c$ est cellulaire.
\end{prop}

\begin{proof} 
 Par \cite[Lemme, p. 103]{Ful}, on a l'\'enonc\'e (2).

Pour (3), par  r\'ecurrence noeth\'erienne, il suffit de montrer que si $X\iso \BA^n$ avec $n\in \BZ_{\geq 0}$, alors $Y^c$ est cellulaire. 
Dans ce cas,  on sait que l'on a $H^1(\BA^n,T)=0$, $Y\iso \BA^n\times T$ 
et donc $Y^c\iso \BA^n\times T^c$. Le r\'esultat d\'ecoule de l'\'enonc\'e (2).

Pour (1), on sait (cf.  \cite[Thm. III.2.3]{K2})  que si la surface $X$ est minimale, alors soit
$X$  est isomorphe   \`a $\BP^2$ soit  $X$  est fibr\'ee en $\BP^1$
au-dessus de $\BP^1$. De telles surfaces sont  cellulaires.
Il suffit  donc de montrer que si une  surface lisse $X$ est cellulaire, 
pour tout  $x\in X(k)$,  la surface  \'eclat\'ee $Y:=Bl_xX$ est cellulaire. 

Supposons que $X=\BA^2\cup Z$ est une d\'ecomposition cellulaire de $X$. 
Si $x\in \BA^2$, il suffit donc de montrer que $Y:=Bl_{(0,0)}\BA^2$ est cellulaire. La vari\'et\'e $Y\sbt \BA^2\times \BP^1$ est d\'efinie par 
l'\'equation $xu=yv$, o\`u 
$\BA^2=Spec\ k[x,y]$ et $\BP^1=Proj\ k[u,v]$. 
Alors $Z(v=0)\iso \BA^1$ et $D(v\neq 0)=Spec\ k[x,y,\frac{u}{v}]/(\frac{u}{v}\cdot x=y)\cong \BA^2$.

Si $x\in Z$, il existe un ouvert $U\sbt X$ et un ferm\'e $V\sbt X$ tels que $U$, $V$  soient cellulaires,
 $U\cap V=\emptyset$, $x\notin U\cup V$ 
 et  $X$ ait une d\'ecomposition cellulaire $X=U\cup \BA^1\cup V$ ou $X=U\cup \BA^0\cup V$.
Ainsi $Y\times_XU$ et $Y\times_XV$ sont cellulaires. 
 Dans le premier cas, $Y\times_X\BA^1=\BP^1\cup \BA^1$ avec $\BP^1\cap \BA^1=\{x'\}$, o\`u $\BP^1$ est le diviseur exceptionnel.
 On a donc $\BP^1\setminus \{x'\}\cong \BA^1$ et $(Y\times_X\BA^1)\setminus (\BP^1\setminus \{x'\})\cong \BA^1$.
Dans le deuxi\`eme  cas, on a  $Y\times_X\BA^0\cong \BP^1\cong \BA^1\cup \BA^0$.
 Le r\'esultat  en d\'ecoule.
\end{proof}

Soit de nouveau $k$ un corps de caract\'eristique z\'ero quelconque.
On utilise dans cet article le complexe motivique $\BZ(n)$ de faisceaux sur les vari\'et\'es lisses sur $k$
(Voevodsky), sous la forme donn\'ee par B. Kahn dans \cite[\S 2]{Ka11}).
Pour toute $k$-vari\'et\'e lisse $X$, dans la cat\'egorie d\'eriv\'ee, on a  $\BZ(n)=0$ pour $n<0$, $\BZ(0)=\BZ$, $\BZ(1)\iso \BG_m[-1]$
et une suite exacte (\cite[Prop. 2.9]{Ka11})
\begin{equation}\label{e1}
\xymatrix{0\ar[r]& CH^2(X)\ar[r]& H^4(X,\BZ(2))\ar[r]& H^3_{nr}(X,\BQ/\BZ(2))\ar[r]& 0.}
\end{equation}
On rappelle un th\'eor\`eme de Bruno Kahn :

\begin{thm}(\cite[Thm. 2.5]{Ka10})\label{tmain1}
Soit $X$ une $k$-vari\'et\'e lisse, int\`egre, g{\'e}om{\'e}triquement cellulaire.
Pour tout entier $n \geq 0$, on a une suite spectrale fonctorielle:
\begin{equation}\label{e2}
E^{p,q}_2(X,n)=H^{p-q}(k,CH^q(X_{\bk})\otimes \BZ(n-q))\Longrightarrow H^{p+q}(X,\BZ(n))
\end{equation}
 et on a un accouplement de suites spectrales :
\begin{equation}\label{e5}
E_r^{p,q}(m)\times E_r^{p',q'}(n)\ra E_r^{p+p',q+q'}(m+n),
\end{equation} 
tel que,  pour  $r=2$, l'accouplement est le cup-produit.
\end{thm}

On trouvera dans l'appendice (\S \ref{accouplements}) des rappels sur l'accouplement de suites spectrales.

La diff\'erentielle
$E^{1,1}_2(X,1)\ra E^{3,1}_2(X,1)$
d\'efinit un homomorphisme :
$$d(1):  \Pic(X_{\bk})^{\Gamma_k} \to  \Br(k).$$
La diff\'erentielle 
$E^{2,2}_2(X,2)\ra E^{4,1}_2(X,2)$
d\'efinit un homomorphisme :
$$d(2):CH^2(X_{\bk})^{\Gamma_k} \to H^2(k,\Pic(X_{\bk})\otimes \bk^{\times}).$$

\begin{lem}\label{lemcont}
Soit $X$ une $k$-vari\'et\'e lisse, g\'eom\'etriquement
int\`egre, g{\'e}om{\'e}triquement cellulaire. 
Alors on a  $\Im(CH^2(X)\to CH^2(X_{\bk})^{\Gamma_k})\sbt \Ker(d(2))$.
\end{lem}

\begin{proof}
Puisque $\BZ(n)=0$ pour $n<0$, 
dans la suite spectrale (\ref{e2}), on a $E_2^{p,q}(X,2)=0$ pour $q>2$. 
Donc on a un morphisme canonique: $H^4(X,\BZ(2))\xrightarrow{d_X} E^{2,2}_{\infty}(X,2)$
 et une inclusion $E^{2,2}_{\infty}(X,2)\sbt E^{2,2}_2(X,2)$. 
Alors on a un diagramme commutatif:
$$\xymatrix{CH^2(X)\ar[r]^{i_X}\ar[d]&H^4(X,\BZ(2))\ar[r]^-{d_X}\ar[d]&E^{2,2}_{\infty}(X,2)\ar[r]\ar[d] &E^{2,2}_2(X,2)=CH^2(X_{\bk})^{\Gamma}\ar[d]\\
CH^2(X_{\bk})\ar[r]^{i_{X_{\bk}}}&H^4(X_{\bk},\BZ(2))\ar[r]^-{d_{X_{\bk}}}&E^{2,2}_{\infty}(X_{\infty},2)\ar[r]
& E^{2,2}_2(X_{\bk},2)=CH^2(X_{\bk}),
}$$
o\`u $CH^2(X)\xrightarrow{i_X}H^4(X,\BZ(2))$ d\'esigne le morphisme dans la suite exacte (\ref{e1}).
Donc 
$$\Im(CH^2(X)\to CH^2(X_{\bk})^{\Gamma_k})\sbt E^{2,2}_{\infty}(X,2)\sbt \Ker(d(2)).$$
\end{proof}

Notons d\'esormais
$\CM(X)$ l'homologie du complexe
\begin{equation}\label{defM}
CH^2(X)\to CH^2(X_{\bk})^{\Gamma_k} \xrightarrow{d(2)} H^2(k,\Pic(X_{\bk})\otimes \bk^{\times}).
\end{equation}
On note
$$ d'(2) : CH^2(X_{\bk})^{\Gamma_k}/\Im CH^2(X) \xrightarrow{d(2)} H^2(k,\Pic(X_{\bk})\otimes \bk^{\times})$$
 l'application induite par $d(2)$. On a $\CM(X) = \ker d'(2)$.

Le th\'eor\`eme suivant g\'en\'eralise \cite[Cor. 7.1 et 7.2]{K96}:

\begin{thm}\label{tmain}
Soit $X$ une $k$-vari\'et\'e lisse, g\'eom\'etriquement
int\`egre, g{\'e}om{\'e}triquement cellulaire. 
Si $H^1(k,\Pic(X_{\bk})\otimes \bk^{\times})=0$, 
 alors les groupes $\CM(X)$ et $\frac{H^3_{nr}(X,\BQ/\BZ(2))}{\Im H^3(k,\BQ/\BZ(2))}$ sont finis et
on a une suite exacte:
\begin{equation}\label{e6}
\xymatrix{0\ar[r]&\frac{H^3_{nr}(X,\BQ/\BZ(2))}{\Im H^3(k,\BQ/\BZ(2))}\ar[r]&\CM(X)\ar[r]&H^4(k,\BQ/\BZ(2)).
}
\end{equation}
\end{thm}

\begin{proof}
Par la suite exacte (\ref{e1}), on a une suite exacte:
$$\xymatrix{CH^2(X)\ar[r]&\frac{H^4(X,\BZ(2))}{\Im H^4(k,\BZ(2))}\ar[r]&\frac{H_{nr}^3(X,\BQ/\BZ(2))}{\Im H^3(k,\BQ/\BZ(2))}\ar[r]&0.
}$$
Dans la suite spectrale (\ref{e2}), on a $E_2^{p,q}(X,2)=0$ pour $q>2$ ou $q<0$ et donc une suite exacte:
$$E^{3,1}_{\infty}(X,2)\to \frac{H^4(X,\BZ(2))}{\Im H^4(k,\BZ(2))}\to \Ker (d(2))\to E^{5,0}_2(X,2).$$
D'apr\`es le Lemme \ref{lemcont}, on a  une suite exacte:
$$
\xymatrix{E^{3,1}_{\infty}(X,2)\ar[r]&\frac{H_{nr}^3(X,\BQ/\BZ(2))}{\Im H^3(k,\BQ/\BZ(2))}\ar[r]&\CM(X)\ar[r]&H^4(k,\BQ/\BZ(2)).
}$$

Si $E^{3,1}_2(X,2)=H^1(k,\Pic(X_{\bk})\otimes \bk^{\times})=0$, on a $E^{3,1}_{\infty}(X,2)=0$ et donc la suite exacte (\ref{e6}).

Par \cite[Lemme 3.3]{K97}, $\Pic(X_{\bk})$ et $CH^2(X_{\bk})$ sont des $\BZ$-modules libres de type fini. 
Puisque $\frac{CH^2(X_{\bk})^{\Gamma_k}}{\Im CH^2(X)}$ est un groupe de torsion, le groupe $\frac{CH^2(X_{\bk})^{\Gamma_k}}{\Im CH^2(X)}$ est fini et donc $\CM(X)$ est fini.
\end{proof}

\begin{rem} 
Pour une $k$-vari\'et\'e lisse g\'eom\'etriquement connexe g\'eom\'etriquement cellulaire,
le groupe $H^0(X_{\bk}, {\mathcal K}_{2})$ est uniquement divisible.  Pour des g\'en\'eralisations
du th\'eor\`eme \ref{tmain} sous cette simple hypoth\`ese, on consultera \cite[Prop. 1.3 et Prop. 2.2]{CT1}. 
\end{rem}

\begin{thm} \label{torseuruniv1}
Soit $X$ une $k$-vari\'et\'e projective, lisse, g\'eom\'e\-triquement int\`egre et g\'eom\'e\-triquement cellulaire.
Soit $\mathcal{T} \to X$ un torseur universel sur $X$ et soit $\mathcal{T}^c$
une $k$-compactification lisse de  $\mathcal{T}$. Alors  le groupe
$H^3_{nr}(\mathcal{T}^c,\BQ/\BZ(2))/H^3(k, \BQ/\BZ(2))$ est fini.
\end{thm}

\begin{proof}
Soit $S$ le $k$-tore de groupe des caract\`eres le r\'eseau $\Pic(X_{\bk})$.
D'apr\`es \cite[Cor. 1]{CTHS}, il existe une $k$-compactification torique lisse $S^c$ de $S$.
Comme le groupe $H^3_{nr}(\mathcal{T}^c,\BQ/\BZ(2))$ est un invariant $k$-birationnel,
il suffit d'\'etablir le r\'esultat pour  $\mathcal{T}^c=\mathcal{T} \times^SS^c$.
D'apr\`es la proposition \ref{universelcellulaire}, $\mathcal{T}^c$ est alors une vari\'et\'e g\'eom\'etriquement cellulaire.
Par ailleurs, le module galoisien $\Pic(\mathcal{T}^c_{\bk})$ est un module de permutation \cite[Thm. 2.1.2]{CTS}.
On a donc $H^1(k, \Pic(\mathcal{T}^c_{\bk}) \otimes \bk^{\times})=0$. Une application du th\'eor\`eme
\ref{tmain} donne alors le r\'esultat.
\end{proof}

D'apr\`es la proposition \ref{universelcellulaire}, le th\'eor\`eme \ref{torseuruniv} est un cas sp\'ecial du th\'eor\`eme \ref{torseuruniv1}.

Pour appliquer le Th\'eor\`eme \ref{tmain} au calcul du groupe $\frac{H^3_{nr}(X,\BQ/\BZ(2))}{\Im H^3(k,\BQ/\BZ(2))}$, on a besoin de contr\^oler  l'application $CH^2(X_{\bk})^{\Gamma_k} \xrightarrow{d(2)} H^2(k,\Pic(X_{\bk})\otimes \bk^{\times})$.

Soit $X$ une $k$-vari\'et\'e lisse, int\`egre, g{\'e}om{\'e}triquement cellulaire.
L'accouplement  (\ref{e5}) pour $n=m=1$ donne   un diagramme commutatif (cf. \S 6) :
$$\xymatrix{E^{1,1}_2(X,1)\otimes E_2^{1,1}(X,1)\ar[r]^-{d_{\otimes}}\ar[d]^{\cup}&
(E^{3,0}_2(X,1)\otimes E_2^{1,1}(X,1))\oplus ( E_2^{1,1}(X,1)\otimes E^{3,0}_2(X,1))\ar[d]^{\cup+\cup}\\
E_2^{2,2}(X,2)\ar[r]^{d(2)}&E_2^{4,1}(X,2),
}$$
o\`u $d_{\otimes}=d(1)\otimes id+ id\otimes d(1)$.
C'est-\`a-dire que l'on a un diagramme commutatif:
\begin{equation}\label{e3}
\xymatrix{\Pic(X_{\bk})^{\Gamma_k}\otimes \Pic(X_{\bk})^{\Gamma_k}\ar[d]^{\cup_1}\ar[r]^-{d_{\otimes}}&
(\Br(k)\otimes \Pic(X_{\bk})^{\Gamma_k})\oplus (\Pic(X_{\bk})^{\Gamma_k}\otimes \Br(k))  \ar[d]^{\cup_2+\cup_2} \\
CH^2(X_{\bk})^{\Gamma_k} \ar[r]^{d(2)}&H^2(k,\Pic(X_{\bk})\otimes \bk^{\times}),
}
\end{equation}
o\`u $\cup_1$ est l'intersection et $\cup_2$ est le cup-produit 
$$H^2(k,\bk^{\times})\times H^0(k,Pic(X_{\bk}))\xrightarrow{\cup_2} H^2(k,\Pic(X_{\bk})\otimes \bk^{\times}).$$

\section{Surfaces de del Pezzo de degr\'e au moins 5}

Une surface projective, lisse, g\'eom\'etriquement connexe $X$ est appel\'ee
 \emph{surface de del Pezzo} si le faisceau anticanonique
 $-K_X$ est ample. Le degr\'e d'une telle surface $X$ est  $deg(X):=(K_X,K_X)$. 
Par \cite[Exercise 3.9]{K2}, $X$ est alors g\'eom\'etriquement rationnelle, on a $1\leq deg(X)\leq 9$ et $Pic(X_{\bk})\iso \BZ^{10-deg(X)}$.
Comme pour toute surface $X$ projective, lisse, g\'eom\'e\-triquement rationnelle, le degr\'e sur les z\'ero-cycles d\'efinit un isomorphisme $CH^2(X_{\bk})\iso \BZ$,
 et on a 
   $$ \frac{CH^2(X_{\bk})^{\Gamma}}{\Im CH^2(X)} = \BZ/I(X),$$ o\`u $I(X)$ d\'esigne l'indice de $X$.

   \medskip

Par les travaux de Enriques, Ch\^atelet,  Manin, 
 Swinnerton-Dyer (voir \cite[Section 4]{CT99} ou \cite[Th\'eor\`eme 2.1]{VA}), on a :

\begin{thm} \label{t2}
Soit $X$ une surface de del Pezzo de degr\'e $\geq 5$. 

(1) Si $X(k)\neq \emptyset$, alors $X$ est $k$-rationnelle;

(2) Si $deg(X)=5 $ ou $7$, alors $X(k)\neq \emptyset$. 
\end{thm}

Soit $X$  une surface de del Pezzo  de degr\'e $\geq 5$. 
Si $X(k)\neq \emptyset$,  l'\'enonc\'e (1) implique que l'on a 
 $\frac{H^{i}_{nr}(X,\BQ/\BZ(j))}{\Im H^{i}(k,\BQ/\BZ(j))}=0$ pour tous entiers $i$ et $j$.
 En particulier  $\Br(X)/\Im \Br(k)=
 \frac{H^2_{nr}(X,\BQ/\BZ(1))}{\Im H^2(k,\BQ/\BZ(1))}=0$
  (voir aussi le Lemme \ref{l1}) et 
 $\frac{H^3_{nr}(X,\BQ/\BZ(2))}{\Im H^3(k,\BQ/\BZ(2))}=0$.
 
 En fait, soit $X$ une surface de del Pezzo de degr\'e au moins $4$, alors $I(X)=1$ implique $X(k) \neq \emptyset$.
 La question analogue est ouvere pour les del Pezzo de degr\'e 3, i.e. les surfaces cubiques.
Ceci n'utilise pas dans le pr\'esent article.

\begin{lem}\label{l1}
Soit $X$ une $k$-surface de del Pezzo de degr\'e $\geq 5$. Alors $\Pic(X_{\bar{k}})$ est stablemement de permutation, $H^1(k,\Pic(X_{\bk})\otimes \bk^{\times})=0$ et $\Br(X)/\Im\Br(k)=0$.
\end{lem}

\begin{proof}
Soit $\CC$ la classe des surfaces $X/K$, pour $K$ corps extension  quelconque de $k$, de del Pezzo de degr\'e $\geq 5$. 
Par le Th\'eor\`eme \ref{t2}, si $X(K)\neq \emptyset$, alors $X$ est $K$-rationnelle, et donc $\Pic(X_{\bar{K}})$ est stablemement  de permutation comme $Gal(\bar{K}/K)$-module (\cite[Proposition 2.A.1]{CTS}). Par \cite[Th\'eor\`eme 2.B.1]{CTS} pour chaque $X/K\in \CC$, le $Gal(\bar{K}/K)$-module $\Pic(X_{\bar{K}})$ est 
stablemement de permutation. Alors $H^1(k,\Pic(X_{\bk})\otimes \bk^{\times})=0$ et $H^1(k,\Pic(X_{\bk}))=0$. 
Puisque $\Br(X_{\bk})=0$,  par la suite spectrale de Hochschild-Serre on obtient 
 $\Br(X)/\Im\Br(k)=0$.
\end{proof}

\begin{prop}\label{exactdP}
Soit $X$ une $k$-surface de del Pezzo de degr\'e $\geq 5$.  On a la suite exacte
\begin{equation}\label{e7}
 0 \to \CM(X) \to \BZ/I(X) \xrightarrow{d'(2)} H^2(k,\Pic(X_{\bk})\otimes \bk^{\times}))
\end{equation}
et la suite exacte
\begin{equation}\label{e8}
\xymatrix{0\ar[r]&\frac{H^3_{nr}(X,\BQ/\BZ(2))}{\Im H^3(k,\BQ/\BZ(2))}\ar[r]&\CM(X)\ar[r]&H^4(k,\BQ/\BZ(2)),
}
\end{equation}
o\`u $\CM(X)$ est l'homologie du complexe (\ref{defM}).
\end{prop}
\begin{proof}
Ceci r\'esulte  du th\'eor\`eme \ref{tmain} et du lemme \ref{l1}.
\end{proof}

\section{Formes tordues de $\BP^1 \times \BP^1$}

Rappelons (voir par exemple \cite[Exemples 3.1.3 et 3.1.4]{AB}) que l'on a :

\begin{prop}\label{degre8}
Soit $X$ une surface de del Pezzo de degr\'e 8 sur un corps $k$. Alors on
a l'une des possibilit\'es suivantes :
 
(1)  $X$ est un \'eclatement de $\BP^2_{k}$ en un $k$-point, et dans ce cas, $X(k)\neq \emptyset$.

(2) Il existe des coniques lisses $C_1$, $C_2$ sur $k$ telles que $X\iso C_1\times C_2$.

(3) Il existe une extension de corps $K/k$ de degr\'e $2$ et une conique $C$ sur $K$ tels que $X\iso R_{K/k}C$, o\`u $R_{K/k}$ d\'esigne la restriction \`a  la Weil de $K$ \`a $k$.

De plus, $\Pic(X_{\bk})$ est un $\Gamma_k$-module de permutation.
\end{prop}

En fait, dans le cas o\`u $X\sbt \BP^3_k$ est un quadrique lisse, on a l'extension discriminant $K/k$ de degr\'e $2$ (peut-\^etre $K=k\times k$)
 et, pour toute section plane lisse $C \subset X$, on a $X\simeq R_{K/k}C_K$.
 Ceci n'utilise pas dans le pr\'esent article.

Dans le cas (2), on a:

\begin{prop}\label{1}
Soient $C_1$, $C_2$ deux   coniques  lisses sur $k$ et $X\iso C_1\times C_2$.
Supposons $X(k)=\emptyset$.  L'image de $d(2)$ est $\BZ/2$. Si $I(X)=2$, alors
$\CM(X)=0$ et $\frac{H^3_{nr}(X,\BQ/\BZ(2))}{\Im H^3(k,\BQ/\BZ(2))}=0$.
Si $I(X)=4$, alors $\CM(X)=\BZ/2$.

\end{prop}

\begin{proof}
On a $\Pic(C_{i,\bk})^\Gamma_k\cong \Pic(C_{i,\bk})\cong \BZ$ pour $i=1,2$. On note $p_i: X\ra C_i$ la projection, et pour $p_i^*: \BZ\cong \Pic(C_{i,\bk})\ra \Pic(X_{\bk})$, on note $e_i:=p_i^*(1_{\BZ})$. 
Alors $\Pic(X_{\bk})^{\Gamma_k}\iso \Pic(X_{\bk})\iso \BZ e_1\oplus \BZ e_2$ et  $H^2(k,\Pic(X_{\bk})\otimes \bk^{\times})\iso \Br(k)e_1\oplus \Br(k)e_2$.

Pour $i=1,2$, on applique \cite[Theorem 4.4 (i)]{K97} \`a $E_2^{1,1}(-,1)\ra E_2^{3,0}(-,1)$.  On obtient  un diagramme commutatif:
$$\xymatrix{\Pic(C_{i,\bk})^{\Gamma_k}\ar[r]^{d(C_i)}\ar[d]^{p_i^*}&  \Br(k)\ar[d]^= \\
\Pic(X_{\bk})^{\Gamma_k}\ar[r]^{d(1)}&\Br(k).
}$$
Notons $[C_i]:=d(C_i)(1_{Pic(C_{i,\bk})})$. En utilisant le diagramme (\ref{e3}), on obtient :
\begin{equation}\label{e4}
d(2)(e_1\cup e_2)=(d(2)\circ \cup_1) 
(e_1\otimes e_2)=\cup_2([C_1]\otimes e_2)+\cup_2([C_2]\otimes e_1)=[C_1]e_2+[C_2]e_1.\end{equation}

On v\'erifie ais\'ement $CH^2(X_{\bk})^{\Gamma_k}\cong CH^2(X_{\bk})\iso \BZ (e_1\cup e_2)$. On obtient :
$\Im (d(2))=0$ si et seulement si $[C_1]=[C_2]=0$ et sinon $\Im (d(2))=\BZ/2$. 
On conclut alors avec la Proposition \ref{exactdP}.
\end{proof}

Dans  le cas (3), le lemme suivant est d\^u \`a Olivier Benoist:

\begin{lem}\label{lem2}
Soient $K/k$ une extension de corps de degr\'e 2, $C$ une conique lisse sur $K$ et $X\iso R_{K/k}C$ avec $X(k)=\emptyset$.
Supposons que $[C]\in \Im (\Br(k)\ra \Br(K))$. Alors $I(X)=2$.
\end{lem}

\begin{proof}
Soit $\alpha\in \Br(k)$ un \'el\'ement tel que $\alpha|_K=[C]\in\Br(K) $.
Puisque $[C]$ est d'indice $2$, l'indice de $\alpha$ est $2$ ou $4$.

Si $\alpha$ est d'indice $2$, il existe une extension $k'$ de degr\'e 2 de $k$ telle que $\alpha|_{k'}=0\in \Br(k')$.

Si $\alpha$ est d'indice $4$, on le repr\'esente par un $k$-corps gauche $D$ de degr\'e $4$. 
Ainsi $D$ est d\'eploy\'e sur une extension $L$ de degr\'e 2 de $K$. 
Alors $D$ contient une sous-alg\`ebre commutative isomorphe \`a $L$, 
donc a fortiori une sous-alg\`ebre commutative isomorphe \`a $K$.
Par un th\'eor\`eme d'Albert (\cite[Thm. 5]{A32}, cf. \cite[Lem. 2.9.23]{J}), 
il existe une extension $k'$ de degr\'e 2 de $k$ telle que $D$ contient une sous-alg\`ebre commutative isomorphe \`a $K':=k'\cdot K$.

Dans tout cas, il existe une extension $k'$ de degr\'e 2 de $k$ telle que $[C]|_{K'}=0\in \Br(K')$, o\`u $K'=k'\cdot K$.
Donc $X(k')\neq \emptyset$ et $I(X)=2$.
\end{proof}

\begin{rem}
Soient $K/k$ une extension de corps de degr\'e 2, $C$ une conique lisse sur $K$ et $X\iso R_{K/k}C$ avec $X(k)=\emptyset$.
Si $[C]\notin \Im (\Br(k)\ra \Br(K))$, alors $I(X)=4$.
Ceci sera montr\'e dans la d\'emonstration de la proposition \ref{2}.
\end{rem}

\begin{prop}\label{2}
Soient $K/k$ une extension de corps de degr\'e 2, $C$ une conique lisse sur $K$ et $X\iso R_{K/k}C$ avec $X(k)=\emptyset$.
Alors $\CM(X)=0$ et $\frac{H^3_{nr}(X,\BQ/\BZ(2))}{\Im H^3(k,\BQ/\BZ(2))}=0$.
\end{prop}

\begin{proof}
 On a $X_K\iso C\times_K C^{\sigma}$, o\`u $\sigma\in Gal(K/k)$, $\sigma\neq id$ et $$C^{\sigma}:=( C\ra Spec\ K\xrightarrow{\sigma} Spec\ K).$$
 Donc $CH^2(X_{\bk})\iso \BZ$ et $I(X)|4$.
 Puisque $d(2)(CH^2(X))=0$, on a $(\#\Im (d(2)))|4$.
 L'hypoth\`ese $X(k)=\emptyset$ \'equivaut \`a $C(K)=\emptyset$.

Par \cite[Theorem 4.4 (i) et (iii)]{K97}, on a un diagramme commutatif
$$\xymatrix{\BZ\cong CH^2(X_{\bk})^{\Gamma_K}\ar[d]^{d(2)_K}\ar[r]^{tr}\ar@{}[rd]|{(1)}
&\BZ\cong CH^2(X_{\bk})^{\Gamma_k}\ar[d]^{d(2)}\ar[r]^{Res}\ar@{}[rd]|{(2)}&
\BZ\cong CH^2(X_{\bk})^{\Gamma_K}\ar[d]^{d(2)_K}\\
H^2(K,\Pic(X_{\bk})\otimes \bk^{\times})\ar[r]^{tr}&H^2(k,\Pic(X_{\bk})\otimes \bk^{\times})\ar[r]^{Res}&H^2(K,\Pic(X_{\bk})\otimes \bk^{\times})
}$$
o\`u $tr$ est le transfert et $Res$ est la restriction. Par la Proposition \ref{1}, l'image de $d(2)_K$ est $\BZ/2$. Par le carr\'e (2), l'image de $d(2)$ est $\BZ/2$ ou $\BZ/4$.
Puisque $tr(1_{CH^2(X_{\bk})})=2\cdot 1_{CH^2(X_{\bk})}$, par le carr\'e (1), l'image de $d(2)$ est $\BZ/2$ si et seulement si
$tr(\Im (d(2)_K))=0$. 

On consid\`ere:
$$\Br(K)\oplus \Br(K)\iso H^2(K,\Pic(X_{\bk})\otimes \bk^{\times})\xrightarrow{tr} H^2(k,\Pic(X_{\bk})\otimes \bk^{\times})\iso \Br(K).$$
Par la d\'efinition du transfert, $tr(a,b)=a+\sigma (b)$ pour chaque $a, b\in \Br(K)$. 
Par l'\'equation \ref{e4}, $d(2)_K(1_{CH^2(X_{\bk})})=([C],[C])$. Donc $tr(\Im (d(2)_K))=0$ si et seulement si  $[C]=\sigma ([C])$, i.e. $[C]\in \Br(K)^{\sigma}$.
Puisque $Gal(K/k)\cong \BZ/2$, on a $H^3(Gal(K/k),K^{\times})\cong H^1(Gal(K/k),K^{\times})=0$, et donc le morphisme $\Br(k)\ra \Br(K)^{\sigma}$ est surjectif.

On a alors: 

(1) si $[C]\in \Im (\Br(k)\ra \Br(K))$, l'image de $d(2)$ est $\BZ/2$ et, par le lemme \ref{lem2}, on a $I(X)=2$;

(2)  si $[C]\notin \Im (\Br(k)\ra \Br(K))$,  l'image de $d(2)$ est $\BZ/4$ et, par le lemme \ref{lemcont}, on a $I(X)=4$.

On conclut alors avec la Proposition \ref{exactdP}.
\end{proof}

\section{Calcul de $\frac{H^3_{nr}(X,\BQ/\BZ(2))}{H^3(k,\BQ/\BZ(2))}$ pour une surface de del Pezzo $X$ de degr\'e $\geq 5$}

Rappelons un fait bien connu.
\begin{lem}\label{l3}
Soit $X$ une $k$-surface de del Pezzo de degr\'e 6. On a:

(1) Il  existe une extension $K_1/k$ de degr\'e divisant 2, une $K_1$-forme $X_1$ de $\BP^2$ sur $K_1$ et un morphisme $f_1: X_{K_1}\ra X_1$ birationnel.

(2) Il existe  une extension $K_2/k$ de degr\'e divisant 3, une surface de del Pezzo $X_2$ de degr\'e 8 sur $K_2$ et un morphisme $f_2: X_{K_2}\ra X_2$ tels que $X_{K_2}$ est un \'eclatement de $X_2$ le long d'un sous-sch\'ema r\'eduit de dimension 0 et de degr\'e 2. Donc l'indice  $I(X_2)$ de la $K_{2}$-surface $X_{2}$
est $1$ ou $2$.
\end{lem}

\begin{proof}
Cela provient du fait que la configuration des 6 courbes exceptionnelles de $X_{\bk}$ est celle d'un hexagone (\cite{CT72}, ou voir \cite[Section 2.4]{VA}).
\end{proof}

\begin{thm}\label{t1}
Soit $X$ une $k$-surface de del Pezzo de degr\'e $\geq 5$. Alors $\CM(X)=\BZ/2$ si et seulement si $I(X)=4$, $deg(X)=8$, et il existe des  coniques lisses $C_1$, $C_2$ sur $k$ telles que $X\iso C_1\times C_2$.

Sinon, $\CM(X)=0$ et donc $\frac{H^3_{nr}(X,\BQ/\BZ(2))}{H^3(k,\BQ/\BZ(2))}=0$.
\end{thm}

\begin{proof}
La derni\`ere implication r\'esulte de la proposition \ref{exactdP}.

Si  $X(k)\neq\emptyset$, 
le morphisme $CH^2(X)\ra CH^2(X_{\bk})=\BZ$ est surjectif.
 Donc $\CM(X)=0$.
Si $deg(X)=5$ ou $7$, par le Th\'eor\`eme \ref{t2}, $X(k)\neq \emptyset$ et donc alors $\CM(X)=0$.
 On  suppose dor\'enavant  $X(k)=\emptyset$.
 
 \medskip

Si $deg(X)=9$ avec $X(k)=\emptyset$,  $X$ est la  vari\'et\'e de Severi-Brauer associ\'ee \`a une alg\`ebre centrale simple $A$ de degr\'e $3$ (cf. \cite[Th\'eor\`eme 1.6]{VA}).
Par un th\'eor\`eme de Bruno Kahn \cite[Th\'eor\`eme 7.1]{K97}, $d(2)(1_{CH^2(X_{\bk})})=2[A]\in Br(k)=H^2(k, \Pic(X_{\bk})\otimes {\bk}^*)$.
Puisque $X(k)=\emptyset$, on a $[A]\neq 0$, $3[A]=0$ et $I(X)=3$.
 Donc $d(2)(1_{CH^2(X_{\bk})})\neq 0$ et $d'(2)$ est injectif. Alors $\CM(X)=0$.

\medskip

Si $deg(X)=8$ avec $X(k)=\emptyset$, 
le r\'esultat en degr\'e 8  est donn\'e par la Proposition \ref{degre8}, la Proposition \ref{1} et  la Proposition \ref{2}.

\medskip
Consid\'erons le cas des surfaces de del Pezzo de degr\'e 6.

S'il existe une surface de del Pezzo $Y$ et un morphisme $f: X\ra Y$ projectif, birationnel, 
alors $f^*$ induit un morphisme des suites spectrales (\ref{e2}) pour $Y$ et $X$.
De plus, $f^*: CH^2(Y_{\bk})^{\Gamma_k}\iso CH^2(X_{\bk})^{\Gamma_k}$ est un isomorphisme 
et $f^*: \Pic(Y_{\bk})\ra\Pic(X_{\bk})$ admet un inverse \`a gauche.
Donc $E^{4,1}_2(Y,2)\xrightarrow{f^*}E^{4,1}_2(X,2)$ est injectif et $ \Ker(d(2)_Y)\xrightarrow{f^*} \Ker(d(2)_X)$ est un isomorphisme.
Donc $\CM(X)\cong \CM(Y)$.

Si $deg(X)=6$, avec $X(k)=\emptyset$, 
par le Lemme \ref{l3} (2), il existe une extension $K_2/k$ de degr\'e divisant
 3 et une surface de del Pezzo $X_2$ de degr\'e 8 sur $K_2$ et un $K_2$-morphisme $f_2: X_{K_2}\ra X_2$ projectif, birationnel, tels que $I(X_2)=1$ ou $2$. 
D'apr\`es ce que l'on a d\'ej\`a \'etabli pour les surfaces de del Pezzo de degr\'e 8, on a $\CM(X_2)=0$ et, d'apr\`es le paragraphe
 ci-dessus, $\CM(X_{K_2})=0$. 
Par \cite[Th\'eor\`eme 4.4 (3)]{K97}, le transfert est bien d\'efini pour la suite spectrale (\ref{e2}). 
Puisque le transfert est bien d\'efini pour la suite exacte (\ref{e1}),
 le transfert est bien d\'efini pour le complexe:
$$CH^2(X)\to CH^2(X_{\bk})^{\Gamma_k}\xrightarrow{d(2)} H^2(k,\Pic(X_{\bk})\otimes \bk^{\times}),$$
et donc le transfert est bien d\'efini pour $\CM(X)$. 
Donc $\CM(X)$ est annul\'e par $3$.

 Par le m\^eme argument (Lemme \ref{l3} (1)) et le r\'esultat en degr\'e 9, le groupe $\CM(X)$ est annul\'e par $2$.
 On a donc $\CM(X)=0$.
\end{proof}

\begin{cor}\label{l2}
Soit $X$ une $k$-surface de del Pezzo de degr\'e $8$ avec $\CM(X)\neq 0$. Si la dimension cohomologique $cd(k)$ de $k$ est $\leq 3$, alors $\CM(X)=\BZ/2$, $I(X)=4$ et
$\frac{H^3_{nr}(X,\BQ/\BZ(2))}{H^3(k,\BQ/\BZ(2))}=\BZ/2$.
\end{cor}

\begin{proof}
Par la Proposition \ref{exactdP},
on a $\frac{H^3_{nr}(X,\BQ/\BZ(2))}{H^3(k,\BQ/\BZ(2))}=\CM(X)$.
Le r\'esultat d\'ecoule du Th\'eor\`eme \ref{t1}. 
\end{proof}

\begin{exam}\label{exampleH3nontrivial}
Soit $k:=\BC(t,x,y)$. Soient $C_1$ la conique correspondant \`a l'alg\`ebre $(t,x)$, $C_2$ la conique correspondant \`a l'alg\`ebre $(t+1,y)$ et $X:=C_1\times C_2$. Alors $\frac{H^3_{nr}(X,\BQ/\BZ(2))}{H^3(k,\BQ/\BZ(2))}=\BZ/2$.
\end{exam}

\begin{proof}
Puisque la dimension cohomologique $cd(k)$ de $k$ est $3$, par le Th\'eor\`eme \ref{t1} et  le Corollaire \ref{l2}, 
il suffit de montrer que $I(X)=4$.
On note $A=(t,x)\otimes (t+1,y)$ l'alg\`ebre de biquaternions. 
Par \cite[Th\'eor\`eme]{Al},  $A$ est un corps gauche si et seulement si, pour chaque point $x_1\in C_1$ de degr\'e $2$ et chaque point $x_2\in C_2$ de degr\'e $2$, on a $k(x_1)\ncong k(x_2)$.
Donc $I(X)=4$ si et seulement si  $A$ est un corps gauche.
Par \cite[Corollaire 4]{CT3}, $A$ est un corps gauche ssi $t$ et $t+1$ sont ind\'ependantes dans $\BC(t)^{\times}/\BC(t)^{\times 2}$, ce qui est satisfait.
\end{proof}

\begin{cor}\label{l4}
Soit $X$ une $k$-surface de del Pezzo de degr\'e $\geq 5$.
Supposons que toute forme quadratique en $6$ variables sur $k$ est isotrope.
Alors $\frac{H^3_{nr}(X,\BQ/\BZ(2))}{H^3(k,\BQ/\BZ(2))}=0$.
\end{cor}

\begin{proof}
D'apr\`es le th\'eor\`eme \ref{t1}, 
il suffit de montrer que, pour toute paire de coniques lisses $C_1$ et $C_2$ sur $k$, on a $I(C_1\times C_2)\neq 4$.
Soient $(a,b)$ l'alg\`ebre de quaternion correspondant \`a $C_1$ 
et $(c,d)$ l'alg\`ebre de quaternion correspondant \`a $C_2$.
Par l'argument de la d\'emonstration de l'exemple \ref{exampleH3nontrivial}, 
$I(C_1\times C_2)=4$ si et seulement si $(a,b)\otimes  (c,d)$ est un corps gauche.
Par un th\'eor\`eme de Albert (cf. \cite[Prop. 1]{CT3}), 
ceci vaut 
si et seulement si la forme quadratique diagonale $<a,b,-ab,-c,-d,cd>$ est anisotrope sur $k$.
Ceci donne imm\'ediatement le r\'esultat annonc\'e.
\end{proof}

\begin{cor}\label{l5}
Soit $X$ une $k$-surface  de del Pezzo de degr\'e $\geq 5$.
Supposons que $k$ satisfait la propri\'et\'e $(C_2)$ (cf. \cite[\S II.4.5]{Se}). 
Alors $H^3_{nr}(X,\BQ/\BZ(2))=0$.
\end{cor}

\begin{proof} 
Par le corollaire \ref{l4} et la d\'efinition de la propri\'et\'e $(C_2)$, on a $\frac{H^3_{nr}(X,\BQ/\BZ(2))}{H^3(k,\BQ/\BZ(2))}=0$.
D'apr\`es \cite[\S II.4.5 Thm. MS]{Se},  la dimension cohomologique $cd(k)$ de $k$ est $\leq 2$.
Alors $H^3(k,\BQ/\BZ(2))=0$ et donc $H^3_{nr}(X,\BQ/\BZ(2))=0$.
\end{proof}

Le th\'eor\`eme  \ref{t1}
donne  la conjecture de Hodge enti\`ere  pour certaines vari\'et\'es de dimension 4 (voir \cite[\S 1]{CTV}) :

\begin{prop}
Soit $X$ une vari\'et\'e  projective et lisse de dimension $4$ munie d'un morphisme dominant $X\xrightarrow{f} S$
de base une surface projective lisse $S$ et de fibre g\'en\'erique
  $X_{\eta}$ une surface   de  del Pezzo de degr\'e $\geq 5$. 
Alors la conjecture de Hodge enti\`ere  en degr\'e $4$ vaut sur $X$.
\end{prop}

\begin{proof}
Puisque $\BC(S)$ satisfait la propri\'et\'e $(C_2)$ (cf. \cite[\S II.4.5]{Se}), d'apr\`es le Corollaire \ref{l5}, $H^3_{nr}(X,\BQ/\BZ(2))=0$.
D'apres Colliot-Th\'el\`ene et Voisin \cite[Thm 3.8]{CTV}, 
il suffit alors de montrer qu'il existe une vari\'et\'e  projective lisse $Y$ de dimension au plus 3 et un morphisme $Y\xrightarrow{f}X$
tels que l'application induite $CH_0(Y)\xrightarrow{f_*} CH_0(X)$ soit surjective.
Comme $X_{\eta}$ est une $\BC(S)$-surface g\'eom\'etriquement rationnelle, il existe une surface $T$ projective et lisse sur $\BC$
et une application g\'en\'eriquement finie $T \to S$, telles que $X_{\eta}\times_{\BC(S)}  \BC(T)$ soit rationnelle sur $\BC(T)$.
 Il existe donc une application rationnelle dominante de $\BP^2\times T$ vers $X$.
 Il existe alors une surface projective et lisse $T'$ birationnelle \`a $T$ et un morphisme $T' \to X$
 tels que l'application induite $CH_0(T')\xrightarrow{f_*} CH_0(X)$ soit surjective.
\end{proof}

\section{Appendice: accouplements de suites spectrales}\label{accouplements}

Soient $X$  une vari\'et\'e lisse sur $k$ et $Sh(X)$ la cat\'egorie des faisceaux \'etales sur $X$. On rappelle quelques d\'efinitions  donn\'ees dans \cite[Section 2.3]{Mc}:

\begin{defi}
Un module bigradu\'e diff\'erentiel de $Sh(X)$ est une collection d'\'el\'ements $E^{p,q}\in Sh(X)$ pour $p,q\in \BZ$ et de morphismes $d: E^{*,*}\ra E^{*,*}$ de bidegr\'e $(s,1-s)$ pour certains $s\in \BZ$, tels que $d\circ d=0$.

Un produit tensoriel de modules bigradu\'es diff\'erentiels $(E^{*,*}(1),  d(1))$, $(E^{*,*}(2), d(2))$ est un module bigradu\'e diff\'erentiel $((E(1)\otimes E(2))^{*,*}, d_{\otimes})$ avec
$$(E(1)\otimes E(2))^{p,q}=\bigoplus_{r+t=p,\ s+u=q}E^{r,s}(1)\otimes E^{t,u}(2)$$
et $d_{\otimes}(x\otimes y)=d(1)(x)\otimes y+(-1)^{r+s}x\otimes d(2)(y)$, o\`u $x\in E^{r,s}(1)$, $y\in E^{t,u}(2)$.
\end{defi}

Pour deux complexes $A$, $B$, par le th\'eor\`eme de K\"unneth, on a un morphisme canonique 
$$\oplus_{s+r=n}H^r(A)\otimes H^s(B)\xrightarrow{p}H^n(A\times B).$$

\begin{defi}
Soient $E_r^{*,*}(1), d_r(1)$, $E_r^{*,*}(2), d_r(2)$ et $E_r^{*,*}(3), d_r(3)$ trois suites spectrales dans $Sh(X)$. 
Un accouplement 
$$\psi: E_r^{*,*}(1)\times E_r^{*,*}(2)\ra E_r^{*,*}(3)$$
est une collection de morphismes $\psi_r: E_r^{*,*}(1)\otimes E_r^{*,*}(2)\ra E_r^{*,*}(3)$ pour chaque $r$, tel que $\psi_{r+1}$ est la composition:
$$E_{r+1}^{*,*}(1)\otimes E_{r+1}^{*,*}(2)\iso H(E_{r}^{*,*}(1))\otimes H(E_{r}^{*,*}(2))\xrightarrow{p} H((E_{r}(1)\otimes E_{r}(2))^{*,*})\xrightarrow{H(\psi_r)} H(E_r^{*,*}(3))\iso E_{r+1}^{*,*}(3)$$
o\`u $p$ est le morphisme dans le th\'eor\`eme de K\"unneth.
\end{defi}

\bigskip

\noindent{\bf Remerciements.}
Nous remercions Jean-Louis Colliot-Th\'el\`ene pour plusieurs discussions.
Je remercie \'e\-galement Olivier Benoist et Bruno Kahn pour leurs commentaires.

\bibliographystyle{alpha}

\begin{thebibliography}{Gro}
\bibitem[A32]{A32} A. A. Albert: \emph{ A note on division algebras of order sixteen}, Bull AMS  38 (1932) 703--706.
\bibitem[A72]{Al} A. A. Albert : \emph{Tensor products of quaternion algebras}, Proc. Amer. Math. Soc. 35 (1972), 65--€"66.
\bibitem[AB]{AB} Asher Auel, Marcello Bernardara: \emph{Semiorthogonal decompositions and birational geometry of del Pezzo surfaces over arbitrary fields}, http://arxiv.org/abs/1511.07576.
\bibitem[CT72]{CT72} Jean-Louis Colliot-Th\'el\`ene: \emph{Surfaces de Del Pezzo de degr\'e 6}, C.R.A.S. Paris t.275 (1972) 109--111.
\bibitem[CT95]{CT95}  Jean-Louis Colliot-Th\'el\`ene: \emph{Birational invariants, purity and the Gersten conjecture}, in  K-Theory and Algebraic Geometry: Connections with Quadratic Forms and Division Algebras, AMS Summer Research Institute, Santa Barbara 1992, ed. W. Jacob and A. Rosenberg, Proceedings of Symposia in Pure Mathematics 58, Part I (1995) 1--64.
\bibitem[CT99]{CT99}  Jean-Louis Colliot-Th\'el\`ene: \emph{Points rationnels sur les vari\'et\'es non de type g\'en\'eral. Chapitre II: Surfaces rationnelles}, notes 1999.
\bibitem[CT02]{CT3}  Jean-Louis Colliot-Th\'el\`ene: \emph{Exposant et indice d'alg\`ebres simples centrales non ramifi\'ees (avec  un appendice par Ofer Gabber)},  L'Enseignement Math\'ematique 48 (2002) 127--146.
\bibitem[CT15]{CT1} Jean-Louis Colliot-Th\'el\`ene: \emph{Descente galoisienne sur le second groupe de Chow: mise au point et applications}, Documenta Mathematica, Extra Volume: Alexander S. Merkurjev's Sixtieth Birthday (2015) 195--220.
\bibitem[CTHS]{CTHS} Jean-Louis Colliot-Th\'el\`ene, David Harari, Alexei N. Skorobogatov: \emph{Compactification \'equivariante d'un tore (d'apr\`es Brylinski et K\"unnemann)}, Expositiones mathematicae 23 (2005) 161--170.
\bibitem[CTS]{CTS} Jean-Louis Colliot-Th\'el\`ene, Jean-Jacques Sansuc: \emph{La descente sur les vari\'et\'es rationnelles II}, Duke Math. J.54 (1987) 375--492.
\bibitem[CTV]{CTV} Jean-Louis Colliot-Th\'el\`ene, Claire Voisin: \emph{Cohomologie non ramifi\'ee et conjecture de Hodge enti\`ere}, Duke Mathematical Journal 161 (2012) 735--801.
\bibitem[Ful]{Ful} William Fulton: \emph{Introduction to toric varieties}, Annals of Mathematics Studies, vol. 131, Princeton University Press, Princeton N.J. (1993).
\bibitem[Ja]{J} Nathan Jacobson: \emph{Finite dimensional division algebras over fields}, Springer-Verlag, Berlin Heidelberg New York, 1996.
\bibitem[K96]{K96} Bruno Kahn: \emph{Applications of weight-two motivic cohomology}, Doc. Math. 1 (1996), 395--416.
\bibitem[Ka97]{K97} Bruno Kahn: \emph{Motivic cohomology of smooth geometrically cellular varieties}, Algebraic K-theory (Seattle, 1997), Proc. Symposia Pure Math. 67, AMS, Providence, 1999, 149--174.
\bibitem[Ka09]{Ka00} Bruno Kahn: \emph{Formes quadratiques sur un corps}, Cours Sp\'ecialis\'es {\bf 15}, Soc. Math. France, Paris, 2009.
\bibitem[Ka10]{Ka10} Bruno Kahn: \emph{Cohomological approaches to $SK_1$ and $SK_2$ of central simple algebras}, 
Documenta Mathematica , Extra Volume: Andrei A. Suslin's Sixtieth Birthday (2010), 317--369.
\bibitem[Ka12]{Ka11} Bruno Kahn: \emph{Classes de cycles motiviques \'etales},  Algebra \&  Number theory 6-7 (2012), 1369--1407.
\bibitem[Koll96]{K2} J\'anos Koll\'ar: \emph{Rational curves on algebraic varieties}, Ergebnisse der Mathematik und ihrer Grenzgebiete (3), vol. 32, Springer, Berlin, 1996.
\bibitem[Ma74]{Ma} Yuri Manin: \emph{Cubic forms: algebra, geometry, arithemic}, North-Holland Publishing Co., Amsterdam, 1974.
\bibitem[Mc]{Mc} John McCleary: \emph{A user's guide to spectral sequences (second edition)}, Cambridge Studies in Advanced Mathematics, vol. 57, Cambridge University Press, 2001.
\bibitem[Se]{Se} J-P. Serre: \emph{Cohomologie galoisenne (cinqui\`eme \'edition, r\'evis\'ee et compl\'et\'ee)}, Lecture Notes in Mathematics 5, Springer-Verlag, Berlin, 1994.
\bibitem[VA]{VA} Anthony V\'arilly-Alvarado: \emph{Arithmetic of del Pezzo surfaces}, Birational geometry, rational curves, and arithmetic (F. Bogomolov, B. Hassett and Y. Tschinkel eds.) Simons Symposia 1 (2013), 293--319.
\end{thebibliography}
\end{document}